\documentclass[a4paper,12pt]{amsart}
\usepackage{amsmath}
\usepackage{graphicx}\usepackage{caption}\usepackage{subcaption}
\usepackage[english]{babel}
\usepackage[usenames]{color}
\newtheorem{theorem}{Theorem}[section]
\newtheorem{proposition}[theorem]{Proposition}
\newtheorem{lemma}[theorem]{Lemma}

\newtheorem{remark}[theorem]{Remark}

\newcommand{\n}{\mathcal{N}}

\newcommand{\dO}{\partial\Omega}

\begin{document}

\title[]{Dirichlet eigenfunctions on the cube, sharpening the Courant nodal inequality}

\author{Bernard Helffer \and Rola Kiwan}
\thanks{B. Helffer : Bernard.Helffer@math.u-psud.fr, Laboratoire de Math\'ematiques,  Univ. Paris-Sud 11 and CNRS, F 91405 Orsay Cedex, France, and \\ Laboratoire Jean Leray,  Universit\'{e} de Nantes.}
\thanks{R. Kiwan: rkiwan@aud.edu, Americain University in Dubai, p.o.box: 28282, Dubai, United Arab Emirates}

%\begin{abstract}
%\end{abstract}

%\keywords{ Dirichlet Laplacian, Nodal set, second eigenfunction, extremal eigenvalue}
%\subjclass[2000]{35P15, 49R50 }
\maketitle

%______________________________________________________
\section{Introduction and Main result}

Consider the Dirichlet eigenvalues of the Laplacian in the domain $\Omega$

\begin{equation}\label{E}
\left\{
\begin{array}{rll}
-\Delta u &= \lambda u & \text{ in } \Omega\\
u&=0& \text{ on } \dO.
\end{array}
\right.
\end{equation}
We denote by $\{\lambda_k \}_{k\ge 1}=\{\lambda_k (\Omega)\}_{k\ge 1}$ the sequence of eigenvalues:
$$  \lambda_1 < \lambda_2 \leq \lambda_3 \leq ... \leq \lambda_k \leq ... $$
It is well known that the first eigenvalue is simple and the eigenfunction $u_1$ has a constant sign in $\Omega$. All the higher order eigenfunctions must change sign inside $\Omega$ and, consequently, must vanish inside $\Omega$. 

We call nodal set of an eigenfunction $u_k$ associated with $\lambda_k$  the closure of the zero set of $u_k$,
$$\n(u_k)=\overline{\{x \in \Omega;\ u_k(x)=0\}}.$$ This nodal set cuts the domain $\Omega \setminus \n(u_k)$ into $\mu_k=\mu(u_k)$ connected components called ``nodal domains''.

The famous Courant nodal  theorem \cite{Cou} of 1923 states  that  $$\mu(u_k) \leq k.$$
We will say that an eigenvalue  $\lambda$ is Courant sharp  if $\lambda=\lambda_k$ and if  there exists an associate  eigenfunction with $k$ nodal domains. If it is always true in the case of dimension $1$ by the Sturm-Liouvillle theory, Pleijel's theorem \cite{Pl} asserts in 1956 that equality can only occur for a finite set of $k$'s, when the dimension is  at least two.\\

Since we know that the first eigenfunction does not vanish and that the second eigenfunction has exactly two nodal domains,  $\lambda_1$ and $\lambda_2$ are Courant sharp ($\mu_1=1$ and $\mu_2=2$). We are now interested in checking if  other eigenvalues are Courant sharp.

Many papers (and some of them quite recent)  have investigated in which cases this inequality is sharp:  Pleijel \cite{Pl}, Helffer--Hoffmann-Ostenhof--Terracini 
 \cite{HHOT1,HHOT2}, Helffer--Hoffmann-Ostenhof \cite{HH1,HH2}, B\'erard-Helffer \cite{BeHe1,BeHe2,BeHe3}, Helffer--Persson-Sundqvist  \cite{HPS}, 
 L\'ena \cite{Len3}, Leydold \cite{Ley1,Ley2,Ley3}. All these results were devoted to $(2D)$-cases in open sets in $\mathbb R^2$ or in surfaces like $\mathbb S^2$ or $\mathbb T^2$.\\

The aim of the current paper is to look for analogous  results for domains in $\mathbb{R}^3$ and, as $\AA.$Pleijel was suggesting (see below for an historical discussion), for  the simplest case of the cube. More precisely, we  will prove:
\begin{theorem}
In the case of the cube $ (0,\pi)^3$ the only eigenvalues of the Dirichlet Laplacian  which are Courant sharp  are the two first eigenvalues: $\lambda_1=3$ and $\lambda_2=6$.
\end{theorem}

\section{Coming back to Pleijel's paper}

Outside the proof of Pleijel's theorem in $2D$,  Pleijel  \cite{Pl} (see also \cite{BeHe1} for a more detailed analysis) considers as an example the case of the square which reads
\begin{theorem}
In the case of the square the only eigenvalues which are Courant sharp for the Dirichlet Laplacian are the two first eigenvalues and the fourth one.
\end{theorem}
The proof was based on a first reduction to the analysis of the eigenvalues less than $68$ (the argument will be extended to the $(3D)$-case below and this is a quantitative version of the proof of Pleijel's theorem), then all the other eigenvalues were eliminated  using this time a more direct consequence of Faber-Krahn's inequality, except three remaining cases for which Pleijel was rather sketchy which have to be treated by hand.\\

At the end of his celebrated paper $\AA$. Pleijel wrote:\\
{\it " In order to treat, for instance the case of the free three-dimensional membrane $[0,\pi]^3$, it would be necessary to use, in a special case, the theorem quoted in \cite{CH}, p. 394\footnote{In the german version, this is p. 454 in the english version. }. This theorem which generalizes part of the Liouville-Rayleigh theorem for the string asserts that a linear combination, with constant coefficients, of the $n$ first eigenfunctions can have at most $n$ nodal domains. However, as far as I have been able to find there is no proof of this assertion in the literature."}\\
 $\AA$. Pleijel was indeed speaking of a result presented in \cite{CH} as being proved in the thesis defended in 1932 at the University of G\"ottingen  by Horst Herrmann (with R. Courant as advisor). This result was never published or confirmed
 and is now called the Courant-Herrmann conjecture \cite{GZ}. Actually, it is said in \cite{GZ} that the authors can not find any mention of the result in the thesis itself. This Courant-Herrmann conjecture was asserting that, for a given $k\in \mathbb N$,  Courant's theorem holds also for linear combinations of eigenfunctions associated with eigenvalues $\lambda_j$ with $j\leq k$. \\ 
  Pleijel is not  explicitly saying why he was needing this result but one could think that he is interested,  because  he speaks about the "free problem" (i.e.  the Neumann 
  problem), in counting the number of components of the restriction of an eigenfunction to a face of the cube $(0,\pi)^3$. 
  Looking for example to the zeroset of
 $$
(x,y,z) \mapsto a \cos x \cos y \cos n z + b \cos y \cos z \cos nx + c \cos z \cos x \cos n y\,,
$$
one gets for fixed $z=0$, a linear combination of the eigenfunctions of the square   $\cos x \cos y$, $\cos y\cos n x$ and $\cos x \cos n y$   corresponding to two different eigenspaces for the Neumann Laplacian in the square $(0,\pi)\times (0,\pi)$. We will not go further in this paper on the Neumann problem but similar questions could also occur in the Dirichlet problem and we typically meet below the eigenfunction $$
(x,y,z) \mapsto a \sin x \sin y \sin n z + b \sin y  \sin z \sin  nx + c \sin z \sin x \sin n y\,,
$$
and will be interested for example  in the intersection of its zero set with the hyperplace $\{z= \frac \pi 2\}$  inside the cube (in the case $n=3$).

\section{Reminder on Pleijel's theorem in $3D$}
Let us first prove that there are only a finite number of eigenvalues that satisfy $\mu_k:=\mu(u_k)=k$. This proof was given in dimension $n$ by B\'erard-Meyer \cite{BeMe}.

\begin{proposition}\label{lupbound}
If $\lambda_k$ is an eigenvalue of (\ref{E}) such that $\lambda_{k-1}< \lambda_k$, and $u_k$ is an associated eigenfunction then:
\begin{equation}\label{upbound}
\lambda_k^\frac 32 |\Omega| \geq \mu(u_k)   \frac 43 \pi^4\,.
\end{equation}
\end{proposition}

\begin{proof}
Assume that the nodal set cuts the domain $\Omega$ in $\mu_k$ connected components and let us denote them by $\Omega_i, 1\leq i \leq \mu_k.$
 Since $u_k$ does not vanish inside $\Omega_i$, it is equal to its first eigenvalue and now using the $(3D)$-Faber-Krahn inequality on each component (see for example  B\'erard-Meyer \cite{BeMe}):  %1982, page 525):
$$
\lambda_k^\frac 32 |\Omega_i| \geq  \frac 43 \pi^4\, \text{ for }  1\leq i \leq \mu_k \,.
$$

Adding together all  the equations we get (\ref{upbound}).
\end{proof}
\begin{theorem}
\begin{equation}\label{pl1}
 \lim\sup_{k\rightarrow +\infty} \frac{\mu_k}{k} \leq \frac{9}{2 \pi^2}< 1\,.
 \end{equation}
In particular, there exists only a finite number of eigenvalues satisfying $\mu_k=k$.
\end{theorem}
\begin{proof}
We start from  the Weyl's asymptotics for the counting function 
\begin{equation}\label{defNlambda}
N(\lambda):=\#\{ k \,, \, \lambda_k < \lambda\},
\end{equation}
which reads
\begin{equation}
N(\lambda) \sim \frac {1}{6 \pi^2}|\Omega|  \lambda^\frac 32 \,.
\end{equation}

For an eigenvalue $\lambda_k$ such that $\lambda_{k-1}<\lambda_k$, we have $N(\lambda_k)=k-1\,.$
Then  from 
$$\lambda_k^\frac 32 \sim \frac {6 \pi^2}{|\Omega|} k  $$
together with (\ref{upbound}), we get \eqref{pl1}.
\end{proof}

\begin{remark}
It is clear from \eqref{pl1} that we cannot have an infinite number of eigenvalues satisfying $\mu_k=k$.
\end{remark}

\section{The case of the cube}
Let us consider  the cube $(0,\pi)\times(0,\pi)\times (0,\pi)$ for which an orthogonal basis of eigenfunctions for the Dirichlet problem is given by: %Later, we will  discuss the 'good' dimensions (irrationals) for which it will be possible to determine the eigenvalues.
$$\left\{
\begin{array}{rl}
u_{\ell,m,n}(x,y,z)& =\sin(\ell x) \cdot  \sin(m y) \cdot \sin(n z)\,, \\
\lambda_{\ell,m,n}&=\ell^2+ m^2+n^2\,,
\end{array}
\right.
$$
 for $\ell, m,n \geq 1$.\\
Applying Proposition \ref{lupbound} for this domain, we get
\begin{proposition}\label{lupboundrect}
If $u_k$ is an eigenfunction associated with $\lambda_k$ such that $u_k$ has $k$ nodal domains and  if $\lambda_{k-1} < \lambda_k$ we have:
\begin{equation}\label{upbound2}
\frac{{\lambda_k}^{\frac 32}} k \geq \frac 43 \pi\,.
\end{equation}
\end{proposition}

Here we will try to find a lower bound for the number $N(\lambda)$, since we know the $\lambda$'s are equal to $\ell^2+ m^2+n^2$ where  $\ell, m, n$ are integers, so we need to count the number of the lattice points of $\mathbb{R}^3$ inside the sphere of radius $\sqrt{\lambda}$.  

\begin{lemma}\label{Lemmalau}
If $\lambda \geq 3$, then 
\begin{equation}\label{lauzz}
N(\lambda) >   \frac \pi 6 \lambda^\frac 32 - \frac{3\pi}{4} \lambda + 3  \sqrt{\lambda -2} -1\,.
\end{equation}
\end{lemma}
The proof is given in the appendix.

\begin{lemma} If $u_k$ is an eigenfunction associated with $\lambda_k$ such that $u_k$ has $k$ nodal domains we have:
\begin{equation}
\left( \frac{3}{4\pi} - \frac\pi 6 \right) \lambda^{\frac 32} + \frac{3\pi}4\lambda- 3\sqrt{\lambda } +3 > 0\,.
\end{equation}
\end{lemma}
\begin{proof}
First by Courant theorem, we have necessarily $\lambda_{k-1} < \lambda_k\,$. \\
Applying \eqref{lauzz}, we have 
$$ 
k-1 = N(\lambda)>\frac\pi 6 \lambda^{\frac 32} -\frac{3\pi}4\lambda +3\sqrt{\lambda -2}-1\,.
$$ 
i.e.
\begin{equation}\label{minoration}
k >\frac\pi 6 \lambda^{\frac 32} -\frac{3\pi}4\lambda +3\sqrt{\lambda -2}\,.
\end{equation}
Together with (\ref{upbound2}), this implies:
\begin{equation*}
\left( \frac{3}{4\pi} - \frac\pi 6 \right) \lambda^{\frac 32} + \frac{3\pi}4\lambda > 3\sqrt{\lambda -2}\,.
\end{equation*}
One immediately sees that, for $\lambda \geq 3\,$,
 $$
 \sqrt{\lambda -2} -\sqrt{\lambda}= -\frac{2}{\sqrt{\lambda} +\sqrt{\lambda -2} }\geq -\frac{2}{1+\sqrt{3}}>-1\,.
 $$
\end{proof}
Now setting $\mu =\sqrt{\lambda}$ we get the third order inequation:
\begin{equation*}
\left( \frac{3}{4\pi} - \frac\pi 6 \right) \mu^3 + \frac{3\pi}4\mu^2- 3\mu +3 > 0\,.
\end{equation*}

Using a calculator we can see that the only real root of the equation 
$\left( \frac{3}{4\pi} - \frac\pi 6 \right) \mu^3 + \frac{3\pi}4\mu^2- 3\mu +3 = 0$ is $\mu=6.97836\,$. This gives that the inequality is true only $0<\mu<6.97836\,,$ hence for $\lambda < 48.7\,.$
So we have finally proved:
\begin{proposition}
 If $u_k$ is an eigenfunction associated with $\lambda_k$ such that $u_k$ has $k$ nodal domains and  if $\lambda_{k-1} < \lambda_k$ , we have:
\begin{equation}
\lambda_k \leq 48\,.
\end{equation}
\end{proposition}

\section{The list}

In this section, we establish the list of the eigenvalues which are less than $48$ and determine which of these eigenvalues satisfy the necessary condition \eqref{upbound2} for being Courant sharp.\\

\begin{tabular}{|p{10cm}|p{3cm}|p{1cm}|}
\hline
 k& $(\ell,m,n) $   & $\lambda_k$  
\\
\hline
$\lambda_1$&(1,1,1)&3\\
\hline
$\lambda_2=\lambda_3=\lambda_4$& (1,1,2) &6\\
\hline
$\lambda_5=\lambda_6=\lambda_7$&$(1,2,2)$&9\\
\hline
$\lambda_8=\lambda_9=\lambda_{10}$&(1,1,3)&11\\
\hline
 $\lambda_{11}$&(2,2,2)&12\\
\hline
$\lambda_{12}=\lambda_{13}=\lambda_{14}=\lambda_{15}=\lambda_{16}=\lambda_{17}$&(1,2,3)&14\\
\hline
$\lambda_{18}=\lambda_{19}=\lambda_{20}$ &(2,2,3)&17\\
\hline
$\lambda_{21}=\lambda_{22}=\lambda_{23}$ &(1,1,4)&18\\
\hline
$\lambda_{24}=\lambda_{25}=\lambda_{26}$ &(1,3,3)&19\\
\hline
$\lambda_{27}=\lambda_{28}=\lambda_{29}=\lambda_{30}=\lambda_{31}=\lambda_{32} $&(1,2,4)&21\\
\hline
$\lambda_{33}=\lambda_{34}=\lambda_{35} $&(2,3,3)&22\\
\hline
$\lambda_{36}=\lambda_{37}=\lambda_{38}$&$(2,2,4)$&24\\
\hline
$\lambda_{39}=\lambda_{40}=\lambda_{41}=\lambda_{42}=\lambda_{43}=\lambda_{44}$&$(1,3,4)$&26\\
\hline
$\lambda_{45}=\lambda_{46}=\lambda_{47}=\lambda_{48}$&$(3,3,3)\ \&\ (1,1,5)$&27\\
\hline
$\lambda_{49}=\lambda_{50}=\lambda_{51}=\lambda_{52}=\lambda_{53}=\lambda_{54}$&$(2,3,4)$&29\\
\hline
$\lambda_{55}=\lambda_{56}=\lambda_{57}=\lambda_{58}=\lambda_{59}=\lambda_{60}$&$(1,2,5)$&30\\
\hline
$\lambda_{61}=\lambda_{62}=\lambda_{63}=\lambda_{64}=\lambda_{65}=\lambda_{66}$& $(1,4,4)\ \&\ (2,2,5)$&33\\
\hline
$\lambda_{67}=\lambda_{68}=\lambda_{69}$&$(3,3,4)$&34\\
\hline
$\lambda_{70}=\lambda_{71}=\lambda_{72}=\lambda_{73}=\lambda_{74}=\lambda_{75}$&$(1,3,5)$&35\\
\hline
$\lambda_{76}=\lambda_{77}=\lambda_{78} $ &$(2,4,4)$&36\\
\hline
$\lambda_{79}=\lambda_{80}=\lambda_{81}=\lambda_{82}=\lambda_{83}=\lambda_{84}=\lambda_{85}=\lambda_{86}=\lambda_{87}$&$(1,1,6)\ \&\ (2,3,5)$&38\\
\hline
$\lambda_{88}=\lambda_{89}=\lambda_{90}=\lambda_{91}=\lambda_{92}=\lambda_{93}=\lambda_{94}=\lambda_{95}=\lambda_{96}$&$(1,2,6)\ \&\ (3,4,4)$&41\\
\hline
 $\lambda_{97}=\lambda_{98}=\lambda_{99}=\lambda_{100}=\lambda_{101}=\lambda_{102}$ & $(1,4,5)$&42\\
\hline
$\lambda_{103}=\lambda_{104}=\lambda_{105}$& $(3,3,5)$&43\\
\hline
$\lambda_{106}=\lambda_{107}=\lambda_{108}$ & $(2,2,6)$&44\\
\hline
$\lambda_{109}=\lambda_{110}=\lambda_{111}=\lambda_{112}=\lambda_{113}=\lambda_{114}$&$(2,4,5)$&45\\
\hline
$\lambda_{115}=\lambda_{116}=\lambda_{117}=\lambda_{118}=\lambda_{119}=\lambda_{120}$& $(1,3,6)$&46\\
\hline
 $\lambda_{121}$&$(4,4,4)$&48\\
\hline
\end{tabular}
\\~\\

Coming back to  the consequences of Faber-Krahn's inequality, one can check that among all the values on the table, the only eigenvalues that satisfy  inequality (\ref{upbound2}) and $\lambda_{k-1}<\lambda_k$ are $\lambda_1\,$, $\lambda_2\,$, $\lambda_5\,$, $\lambda_8\,$ and $\lambda_{12}\,$.
\begin{proposition}
The only eigenvalues which can be "Courant sharp" are the eigenvalues $\lambda_k$ with $k=1,2, 5, 8$ and $12$.
\end{proposition}
As $ \lambda_1$ and $\lambda_2$ are Courant sharp, the only remaining cases to analyze correspond to $k=5,8,12$.

In the next section we will by a finer analysis involving symmetries eliminate other cases.

\section{Courant theorem with symmetry}
We first recall some generalities which come back to Leydold \cite{Ley1}, and were used in various contexts \cite{Ley2,Ley3, HPS,HHOT2}. Suppose that there exists an isometry $g$ such that $g (\Omega) = \Omega$ and $g^2 = {\sf Id}\,$. Then $g$ acts naturally on $L^2(\Omega)$ by $gu ({\bf x}) = u(g^{-1} {\bf x})\,,\, \forall {\bf x} \in \Omega\,,$ and one can naturally define an orthogonal decomposition of $L^2(\Omega)$
$$ L^2(\Omega)= L^2_{\sf odd} \oplus L^2_{\sf even}\,,$$
where by definition $L^2_{\sf odd}= \{u\in L^2 \,,\, gu =-u\}$, resp. $L^2_{\sf even}= \{u\in L^2 \,,\, gu =u\}$. These two spaces are left invariant by the Laplacian and one can consider separately the spectrum of the two restrictions. Let us explain for the ``odd case'' what could be a Courant theorem with symmetry. If $u$ is an eigenfunction in $L^2_{\sf odd}$ associated with $\lambda$, we see immediately that the nodal domains appear by pairs (exchanged by $g$) and following the proof of the standard Courant theorem we see that if $\lambda = \lambda^{\sf odd}_j$ for some $j$ (that is the $j$-th eigenvalue in the odd space), then the number $\mu(u)$ of nodal domains of $u$ satisfies $\mu(u) \leq j$.\\ 
We get a similar result for the "even" case (but in this case a nodal domain $D$ is either $g$-invariant or $g(D)$ is a distinct nodal domain).\\ 
These remarks may lead to improvements when each eigenspace has a specific symmetry. As we shall see, this will be the case  for the cube with the map $(x,y,z)\mapsto (\pi-x,\pi -y,\pi-z)$.\\
We observe indeed  that
$$
u_{\ell,m,n}(\pi -x,\pi -y,\pi -z) =  (-1)^{\ell + m + n +1} u_{\ell,m,n} (x,y)\,,
$$
and that
$$
\ell^2 +m^2 + n^2 \equiv \ell + m+ n \, ({\rm mod}. 2)\,.
$$
Hence, for a given eigenvalue the whole eigenspace is  even if $\ell +m+n$ is odd and  odd if $\ell+m+n$ is even. Equivalently, the whole eigenspace is  even if the eigenvalue is odd and even if the eigenvalue 
 is odd.\\

{\bf Application.}\\
{\bf  $\lambda_5$ is not Courant sharp.}  The eigenspace  associated with $\lambda_5=9$ is even. This is the second one (in this even space). Hence it should have less than four nodal domains by Courant's theorem with symmetry and has labelling $5$.\\

{\bf  $\lambda_{12}$ is not Courant sharp.}

 $\lambda_{12}=14$ is the fifth eigenvalue in the odd space with respect to $\sigma$. It should has less than $10$ nodal domains  and has labelling $12$.

\section{The remaining value: $k=8$}
\subsection{Main result}
The proof of our main theorem relies now on the analysis of the last case which is the object of the next proposition.
\begin{proposition}
In the eigenspace associated with $\lambda_8$ the eigenfunctions have either $2$, $3$ or $4$ nodal domains. In particular $\lambda_8$ cannot be Courant sharp.
\end{proposition}

\subsection{Preliminaries}~\\
For the value $\lambda_8 =11$ we have to analyze the zeroset of 
$$
\Phi_{a,b,c} (x,y,z):=a \sin x \sin y \sin 3z + b \sin y \sin z \sin 3x + c \sin z \sin x \sin 3y\,,
$$
for $(a,b,c) \neq (0,0,0)$\,.\\

This looks nice because we can divide by $\sin x \sin y \sin z$ and by making the  change of coordinates $u=\cos x$, $v=\cos y$, $w=\cos z$, we get for the zero set of $\Phi_{a,b,c}$  in the new coordinates 
 a quadric surface $\mathcal{Q}_{a,b,c}$ to analyze in the cube  $\mathcal{C}=(-1,1)^3\,$, whose equation is
\begin{equation*}
( \mathcal{Q}_{a,b,c})  \quad 4 \, (a u^2+ b v^2 + c w^2) - (a+b+c) =0\,,  
 \end{equation*}
for $(a,b,c) \neq (0,0,0)$\,.\\

When $a+b+c \neq 0$, we immediately see that there are no critical points inside the cube, so the nodal set is simply an hypersurface (cylinder, ellipsoid or hyperboloid with one or two sheets). In this case, this is  the analysis at the six faces of the cube  which will be decisive for analyzing possible changes  in the number of connected components. 
 In the case when $a+b+c =0$, we have a double cone with a unique critical point at $(0,0,0)$. \\
 In the next subsections, we discuss the different cases.\\

\subsection{Cylinder}~\\
This corresponds to the case  $abc=0\,$. We can use the (2D)- analysis as done in \cite{BeHe1}. It is known that the number of nodal domains can only be 2,3 or 4 (See Section 3.1 and  figure 2.1 there). See figure \ref{fig:cyl}.\\

\begin{figure}[h]
\begin{subfigure}{.3\textwidth}
\includegraphics[width=.8\linewidth]{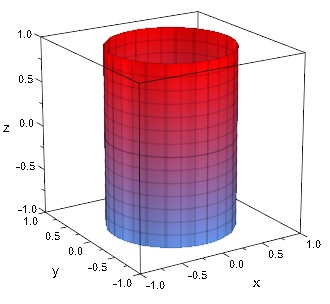}
\caption{\small ~\\ $(a,b,c)=(1,1,0)$}
\end{subfigure}
\begin{subfigure}{.3\textwidth}
\includegraphics[width=.8\linewidth]{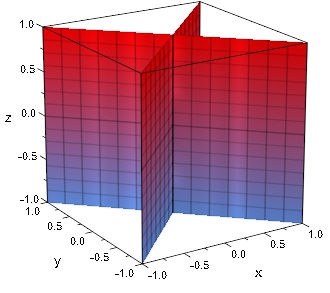}
\caption{\small ~\\ $(a,b,c)=(1,-1,0)$}
\end{subfigure}
\begin{subfigure}{.3\textwidth}
\includegraphics[width=.8\linewidth]{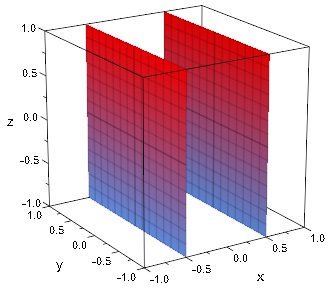}
\caption{\small ~\\ $(a,b,c) =(1,0,0)$}
\end{subfigure}
\caption{Cylinders}
\label{fig:cyl}
\end{figure}

\subsection{Double cone}~\\
 This corresponds to $abc \neq 0$ and $a+b+c=0\,$. The equation of $\mathcal{Q}_{a,b,c}$ is: $$au^2+bv^2=-cw^2\,.$$
  One can verify that the intersection of this cone with each horizontal side $w=\pm 1$ is exactly at the vertices of the cube $u^2=v^2=1$, and that the intersection with each vertical face is a hyperbola. Therefore there are three connected components of $\mathcal{C}\setminus \mathcal{Q}_{a,b,c}$. See figure \ref{fig:cone}.\\
  
\begin{figure}[h]
\includegraphics[width=.3\linewidth]{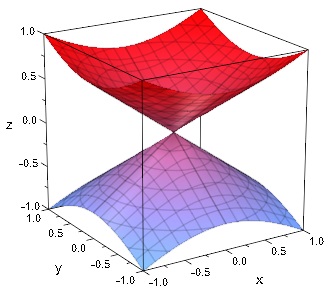}
\caption{Double cones. \small{$(a,b,c)=(0.2,0.2,-0.4)$}}
\label{fig:cone}
\end{figure}

 \subsection{Ellipsoid}~\\
This corresponds to $abc \neq  0$, with  a, b, c of the same sign.
Without loss of generality, we can assume   that $0< a\leq b\leq c$ and 
that  $a+b+c=1$. We note that this implies $\frac 32 (a+b) \leq a + 2b \leq 1$.
\begin{equation}\label{abc}
a u^2+ b v^2 + (1-a-b)w^2 =\frac 1 4\,.
\end{equation}
We denote by $\Omega_{a,b,c}$ the open full ellipsoid delimited by $\mathcal Q_{a,b,c}$.\\
Let us look at the intersection of $\mathcal Q_{a,b,c}$ with the horizontal faces. We have
$$
a u^2 + b v^2 \leq -\frac 34 + \frac 23< 0\,.
$$
We deduce that in this case there are no possible intersections with the horizontal faces, and therefore two subcases can occur depending on the intersection of $\mathcal Q_{a,b,c}$ with the vertical edges. This set is determined by
\begin{equation}\label{abw=1}
 (1-a-b)w^2 =\frac 1 4 - (a+b)\,,\, w \in (-1,+1)\,.
\end{equation}
See figure \ref{fig:ell}.\\
{\bf Subcase} $(a+b) > \frac 14\,$. \\  The ellipsoid $\mathcal{Q}_{a,b,c}$ does not touch the vertical edges and in this case  $\mathcal{C} \cap   \overline{\Omega_{a,b,c} }^{\,c} $  is connected and $\mathcal{C}\setminus \mathcal{Q}_{a,b,c}$ has exactly  two connected components.\\
{\bf Subcase} $(a+b) \leq \frac 14\,.$ \\ $\mathcal{Q}_{a,b,c}$ cuts each vertical edge along a segment $[-w_0,+w_0]$ with  $w_0 = \sqrt{\left( \frac 1 4 - (a+b)\right) / (1-a-b)}$.  The intersection of $\mathcal{Q}_{a,b,c}$ with each vertical face of the cube is the union of two arcs of an ellipse.
 In this case it is clear that $\mathcal{C} \setminus \mathcal{Q}_{a,b,c}$ has three connected components. \\

\begin{figure}[h]
\begin{subfigure}{.3\textwidth}
\includegraphics[width=.8\linewidth]{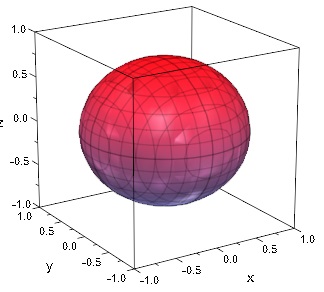}
\caption{\small $(a,b,c)\\= (0.3,0.3,0.4)$}
\end{subfigure}
\begin{subfigure}{.3\textwidth}
\includegraphics[width=.8\linewidth]{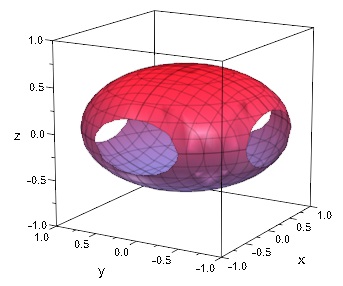}
\caption{\small $(a,b,c)\\= (0.2,0.2,0.6)$}
\end{subfigure}
\begin{subfigure}{.3\textwidth}
\includegraphics[width=.8\linewidth]{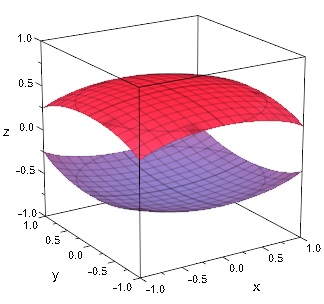}
\caption{\small $(a,b,c)\\=(0.1,0.1,0.8)$}
\end{subfigure}
\caption{Ellipsoids}
\label{fig:ell}
\end{figure}

\subsection{One sheet hyperboloid }~\\
This corresponds to $abc\neq 0$, a,b,c not of the same sign and $(abc) (a+b+c) < 0\,$. 
Without loss of generality, we can assume that $b \geq a  >0$,  $c <0$ and $a+b+c=1\,$.
We note that this implies that $\mathcal Q_{a,b,c}\cap \{w=0\}$ is an ellipse contained in the cube.\\
 
The equation of $\mathcal{Q}_{a,b,c}$ can be written as: $$ au^2+bv^2=\frac 14 -cw^2\,.$$
$\mathcal Q_{a,b,c}$ cuts $\mathbb R^3$ into two components   $\Omega_{a,b,c}^+$ and $\Omega_{a,b,c}^-$ where $\Omega_{a,b,c}^+$ contains $(0,0,0)$. But we have to look inside the cube.\\

  We first observe that $\mathcal Q_{a,b,c}$ has empty intersection the vertical edges. We have indeed 
$$
a + b = 1-c > \frac 14 - c \geq \frac 14 - c w^2\,.
$$
We now look at the intersection with $w=0$. We get an ellipse \break $\mathcal E^0_{a,b,c}:= \mathcal Q_{a,b,c} \cap \{ w=0\}$, whose equation is 
$$
au^2 + b v^2 = \frac 14\,.
$$
We observe that this ellipse could be included in the cube, if $a > \frac 14$ or not if $a \leq \frac 14$.

We also look at the intersection with the upper horizontal face.  We note that the ellipse  $\mathcal E^1_{a,b,c}:= \mathcal Q_{a,b,c} \cap \{w=1\}$  has always  a non empty intersection with this face. \\

Four  subcases appear (See figure \ref{fig:hyp}):\\

\begin{figure}[h]
\begin{subfigure}{.25\textwidth}
\includegraphics[width=.9\linewidth]{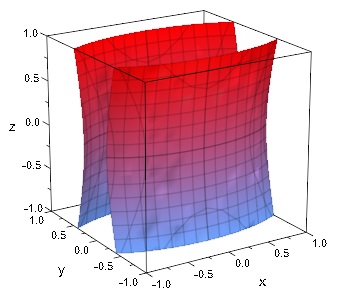}
\caption{\small $(a,b,c)\\= (0.2, 0.9, 0.1)$}
\end{subfigure}
\begin{subfigure}{.25\textwidth}
\includegraphics[width=.9\linewidth]{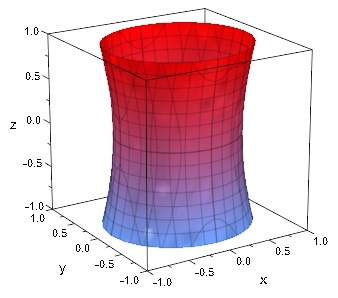}
\caption{\small $(a,b,c)\\ =(0.5,0.6,-0.1)$}
\end{subfigure}
\\
\begin{subfigure}{.25\textwidth}
\includegraphics[width=.9\linewidth]{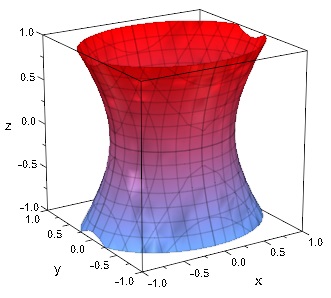}
\caption{\small $(a,b,c)\\=(0.5, 0.8,-0.3)$}
\end{subfigure}
\begin{subfigure}{.25\textwidth}
\includegraphics[width=.9\linewidth]{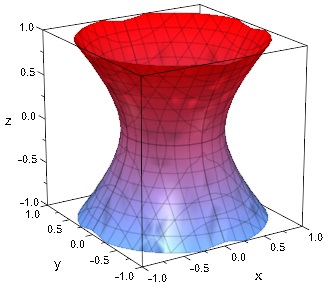}
\caption{\small $(a,b,c)\\ =(0.8,0.8, -0.6)$}
\end{subfigure}
\caption{One Sheet Hyperboloid}
\label{fig:hyp}
\end{figure}

{
\bf Subcase $a \leq \frac 14$}\\
Under this condition  $\mathcal Q_{a,b,c}\cap \{v=0\}\cap \mathcal C$ is empty. Hence  $ \{v=0\}\cap \mathcal C$  is contained in one nodal domain which is invariant by $v\mapsto -v$. The other nodal domains are exchanged by this symmetry. This gives an odd number of nodal domains and this can not be Courant sharp because the labelling is $8$. More precisely the two curves in $\mathcal C$ defined by:
$$
v = \pm  \sqrt{\frac{\frac 14 - a u^2 - c w^2}{b}}
$$
cut the cube in three components. \\

The three last subcases are under the condition that $a >\frac 14$. We note that this condition implies that $\mathcal E^0_{a,b,c}$ is strictly included in the square $(-1,+1)\times (-1,+1)$ and the discussion continues according to the position of $\mathcal E^1_{a,b,c}$ in the horizontal face.\\

{\bf Subcase $\frac 14 <a \leq b<\frac 34$} \\
$\mathcal E^1_{a,b,c}$ is contained in the horizontal face and $\mathcal Q_{a,b,c}$ cuts the cube in two connected domains. \\

{\bf Subcase $\frac 14 < a < \frac 34 \leq b$} \\
$\mathcal E^1_{a,b,c} \cap \partial \mathcal C$ consists of two curves but $\mathcal Q_{a,b,c}$  continue to cut the cube in two domains. For joining two points of $\Omega^-_{a,b,c} \cap \partial \mathcal C$ one can always go to a point in $\{w=0\}$ outside of $\mathcal E^0_{a,b,c}$ and use the connexity (inside the square $\mathcal C$  $\cap$ $\{w=0\}$) of the complementary of the full ellipse.\\

  {\bf Subcase $\frac 34 \leq a$} \\
  $\mathcal E^1_{a,b,c} \cap \mathcal C$ consists of four  curves.  $\mathcal Q_{a,b,c}$  continue to cut the cube in two domains.\\

\subsection{Two sheets hyperboloid}~\\
This corresponds to  $abc\neq 0$, a,b,c not of the same sign and   $(abc) (a+b+c) >0$\,.\\
We can assume $b \geq a  >0$,  $c <0$ and $a+b+c=-1\,.$
The equation of $\mathcal{Q}_{a,b,c}$ can be written as: $$au^2+bv^2= -\frac 14 -cw^2\,.$$
The hyperplane $\{w=0\}$ is contained in one connected component. Hence looking at the symmetry $w \mapsto -w$, we get that necessarily an odd number ($\geq 3$) of nodal domains and $\leq 8$ by Courant's theorem. Hence we know that it cannot be Courant sharp.\\
More precisely, $\mathcal{Q}_{a,b,c}$  meets the hyperplane $\{w=1\}$ along the ellipse $\mathcal E_{a,b,c}$ which this time contains the horizontal upper face of the cube. The analysis of the intersection along each of the vertical faces
 (two symmetric curves by $w\mapsto -w$) shows 
 that we always have exactly three connected components.
See figure \ref{fig:2hyp}.

\begin{figure}[h]
\includegraphics[width=.3\linewidth]{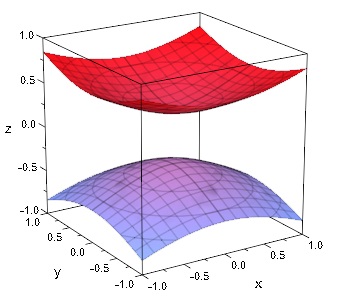}
\caption{Two Sheets Hyperboloid. \small $(a,b,c)=(0.8,0.8,=-2.6)$}
\label{fig:2hyp}
\end{figure}

\section{Conclusion}
In this paper we have analyzed the problem in the simplest example proposed by $\AA$. Pleijel.
 One can of course ask for similar questions for other geometries starting with the parallelepipeds, the ball, the flat tori... The situation for
  $(0,\alpha\pi) \times([0,\beta \pi)\times (0,\gamma \pi)$ is in principle easier in the "irrational" case when 
  $\alpha \ell^2 + \beta m^2 + \gamma n^2 = \alpha \ell_1^2 + \beta m_1^2 + \gamma_1 n_1^2$ implies $(\ell,m,n) = (\ell_1,m_1,n_1)$. Each eigenvalue $\alpha \ell^2+\beta m^2 + \gamma n^2$ 
   is indeed of multiplicity $1$ and the corresponding eigenfunction has $\ell m n$ nodal domains.\\
  One can also think of analyzing  "thin structures" (for example $\gamma$ small or $\beta+\gamma$ small, where previous results in lower dimension can probably be used) in the spirit of \cite{HH1} 
  and get partial results. Another interesting question would be to analyze the Neumann problem for the cube in the spirit of \cite{HPS}.\\

{\bf Acknowledgements}\\
The pictures were obtained by using the programme MATLAB.
This paper was achieved when the two authors were invited at  the Isaac Newton Institute for mathematical sciences in Cambridge. Moreover the first author was supported there as  Simons foundation visiting fellow.
%----------------------
\bibliographystyle{plain}
%\bibliography{mybib2010}

\appendix \section{The proof of Lemma \ref{Lemmalau}}
We follow an idea appearing in the $(2D)$ case in a course of R. Laugesen \cite{Lau}. We start by assuming that $\lambda$ is not an eigenvalue.  With each triple $(\ell,m,n)$ with $\ell \geq 1$, $m\geq 1$, $n\geq 1$, we associate the cube $$Q_{\ell, m,n} =[\ell -1,\ell]\times [m-1,m]\times [n-1,n]\,.$$
 We observe that
$$
N(\lambda) = \sum _{\ell^2+ m^2+n^2 < \lambda, \ell\geq 1,m\geq 1,n\geq 1} A(Q_{\ell, m,n}) \leq \frac \pi 6 \lambda^\frac 32\,.
$$
We are interested in the lower bound. The claim of Laugesen is that
\begin{equation}\label{lau1}
N(\lambda) >  A ( B_\lambda)\,,
\end{equation}
where
$$
B_\lambda:= \{ (x+1)^2 + (y+1)^2 + (z+1)^2< \lambda, x>0, y >0, z>0\}.
$$
The observation is that
$$
B_\lambda \subset  \cup_{\ell^2+m^2+n^2 < \lambda, \ell\geq 1, m\geq 1,n\geq 1}Q_{\ell,m,n} \,.
$$
For $t >0$, $[t]_+$ denotes the smallest integer $\geq t$.\\
Let $(x,y,z)\in B_\lambda$, then it is immediate to see that $(x,y) \in Q_{[x]_+,[y]_+,[z]_+}$. It remains to verify that
$$ Q_{[x]_+,[y]_+,[z]_+} \subset D(0, \sqrt{\lambda})\,.
$$
But we have, for $(x,y,z) \in B_\lambda$, 
$$
[x]_+^2 + [y]_+^2 + [z]_+^2  \leq (x+1)^2 + (y+1)^2 + (z+1)^2 < \lambda\,.
$$
Coming back to \eqref{lau1}, we have to find a lower bound for the area of $B_\lambda$. We note that by translation by the vector $(1,1,1)$:
\begin{equation}
A(B_\lambda) = A( C_\lambda)\,,
\end{equation}
where $$C_\lambda := D(0, \sqrt{\lambda}) \cap \{x >1 \} \cap\{y >1\}\cap\{z>1\}\,.$$
Let $\chi$ the characteristic function of the interval $(0,1)$. We have to compute the integral
$$
A(C_\lambda) = \int _{D(0,\sqrt{\lambda})} (1-\chi (x)) (1-\chi(y) )(1-\chi(z)) dx dy dz\,.
$$
Developing the formula and using the symmetry by permutation of the variables, we get, if $\lambda \geq 3$,
\begin{equation}
\begin{array}{ll}
A(C_\lambda) & = \quad  \int _{D(0,\sqrt{\lambda})} dx dy dz\\&\quad  - 3 \int _{D(0,\sqrt{\lambda})} \chi (x) dx dydz  \\
 &\quad  + 3  \int _{D(0,\sqrt{\lambda})} \chi (x) \chi(y)dx dydz\\ & \quad  -  \int _{D(0,\sqrt{\lambda})} \chi (x)\chi(y)\chi(z)  dx dydz \,.
\end{array}
\end{equation}
It is then immediate to get the lemma by observing that
\begin{equation}\label{lau4}
\int_{D(0,\sqrt{\lambda})} \chi (x) \chi(y)) dx dy dz >   \sqrt{\lambda-2}\,.
\end{equation}
We have assumed till now that $\lambda$ was not an eigenvalue. But if $\lambda$ is an eigenvalue $>3$, we can apply the previous result for an increasing sequence $\hat \lambda _j$ such  that $\hat \lambda_j \rightarrow \lambda$ (where $\hat \lambda_j >3$ is not an eigenvalue). According to our definition of $N(\lambda)$ in \eqref{defNlambda}, we can pass to the limit and observing that in \eqref{lau4} the inequality is uniformly strict when applied to the sequence $\lambda_j$ , we keep the strict inequality when passing to the limit.  The case $\lambda=3$ can be verified directly.
\\
\end{document}